\documentclass[12pt]{article}
\usepackage{graphicx}
\usepackage{gensymb}
\usepackage{mathtools}
\usepackage{mathrsfs}
\usepackage{wrapfig}
\usepackage{hyperref}
\usepackage{mathtools}
\usepackage{textcomp}
\usepackage{pdfpages}
\usepackage{pgfplots}
\usepackage{amsmath}
\usepackage[utf8]{inputenc}
\usepackage{t1enc}
\usepackage{amssymb}
\usepackage{amsthm}
\usepackage{verbatim}
\usepackage{graphicx}
\usepackage{xr}
\usepackage{subcaption}
\usepackage{caption}
\usepackage{float}

 \newtheorem{theorem}{Theorem}[]
 \newtheorem{lemma}[theorem]{Lemma}
 \newtheorem{corollary}[theorem]{Corollary}
\newtheorem{observation}[theorem]{Observation}
 
 \newtheorem{conjecture}[theorem]{Conjecture}
 \newtheorem{claim}[theorem]{Claim}
 \newtheorem{definition}[theorem]{Definition}

\newcommand\cref[1]{Corollary~\ref{cor:#1}}

\newcommand\cS{\mathcal{S}}

\begin{document}

\title{Convex Hull Thrackles}

\author{Bal\'azs Keszegh\thanks{Alfréd Rényi Institute of Mathematics and ELTE Eötvös Loránd University, Budapest, Hungary. 
		Research supported by the J\'anos Bolyai Research Scholarship of the Hungarian Academy of Sciences, by the National Research, Development and Innovation Office -- NKFIH under the grant K 132696 and FK 132060, by the \'UNKP-21-5 and \'UNKP-22-5 New National Excellence Program of the Ministry for Innovation and Technology from the source of the National Research, Development and Innovation Fund and by the ERC Advanced Grant ``ERMiD''. This research has been implemented with the support provided by the Ministry of Innovation and Technology of Hungary from the National Research, Development and Innovation Fund, financed under the  ELTE TKP 2021-NKTA-62 funding scheme.}
\and D\'aniel Simon\thanks{Alfréd Rényi Institute of Mathematics. Research supported by the ERC Advanced Grant ``Geoscape''.}}

\maketitle
\begin{abstract}
	A \emph{thrackle} is a graph drawn in the plane so that every pair of its edges meet exactly once, either at a common end vertex or in a proper crossing. Conway's thrackle conjecture states that the number of edges is at most the number of vertices. It is known that this conjecture holds for linear thrackles, i.e., when the edges are drawn as straight line segments.
	
	We consider \emph{convex hull thrackles}, a recent generalization of linear thrackles from segments to convex hulls of subsets of points. We prove that if the points are in convex position then the number of convex hulls is at most the number of vertices, but in general there is a construction with one more convex hull. On the other hand, we prove that the number of convex hulls is always at most twice the number of vertices.
\end{abstract}

A \emph{drawing of a graph} is a representation of the graph in the
plane such that the vertices are represented by distinct points and the edges by simple continuous curves connecting the corresponding point pairs and not passing through any other point representing a vertex. When it leads to no confusion, we identify the vertices and the edges with their representations.
A drawing of a graph is a \emph{thrackle} if every pair of edges meet precisely once, either at a common endvertex or at a proper crossing. Conway’s conjecture \cite{woodall1971thrackles} states that every thrackle of $n$ vertices can have at most $n$ edges.

It is easy to draw thrackles with $n$ edges but from above, through a series of improvements (Lov\'asz, Pach, and Szegedy \cite{Lovasz1997} proved $2n$, Cairns and Nikolayevsky proved $\frac{3}{2} n$ \cite{cairns2000bounds}, improved further by Fulek and Pach \cite{fulek2011computational,fulek2019thrackles} and by Goddyn and Xu \cite{goddyn2017bounds}), the current best upper bound is $1.393n$ due to Xu \cite{xu2021new}. 

Perhaps the most natural special case is the case of \emph{linear thrackles}, i.e., when edges are represented by straight line segments. The conjecture was originally proved in this case by Erd\H os \cite{erdos1946sets}. Another nice proof of this is due to Perles, see \cite{pach2011conway}.

Further natural special cases solved are when the edges are represented by $x$-monotone curves \cite{pach2011conway} (this generalizes linear thrackles) and when the drawing is \emph{outerplanar} \cite{cairns2012outerplanar}, i.e., if the points lie on the boundary of a disk and the edges lie inside this disk. In the latter case, the proof method of Perles again works while in the former it fails \cite{pach2011conway}.

There are plenty of generalizations of thrackles that we do not consider here in detail (when we allow tangencies, when we only require an odd number of intersections between edges, etc.).

The version that interests us is the following. \'Agoston et al. proposed a generalization of linear thrackles \cite{agoston2022orientation}, which was altered by a suggestion of Gosset to get the following variant \cite{pc} (see Section \ref{sec:disc} for more details).

\begin{definition}
Suppose that we are given a set $P$ of $n$ points in general position in the plane, and a family $\cS$ of subsets of $P$ whose convex hulls are all different and form the family $C(\cS)=\{Conv(S):S\in\cS\}$ where $Conv(S)$ denotes the convex hull of $S$. We say that $C(\cS)$ is a convex hull thrackle on $P$ if the following hold:
    \begin{enumerate}
        \item $C_1\not\subset C_2$ for any two different $C_1,C_2\in C(\cS)$;
        \item $C_1\cap C_2\ne\emptyset$ for any two different $C_1,C_2\in C(\cS)$;
        \item $C_1\cap C_2\cap C_3\subset P$ for any three different $C_1,C_2,C_3\in C(\cS)$.
    \end{enumerate}        
\end{definition}

Note that by the definition $C_1\cap C_2\cap C_3$ is either empty or a single point of $P$.

Observe that if in a convex hull thrackle all convex hulls are segments then it is a linear thrackle and, in the other direction, a linear thrackle is a convex hull thrackle in which all convex hulls are segments.

\begin{conjecture}\cite{agoston2022orientation}\label{conj}
 A convex hull thrackle on $n$ points has at most $n$ convex hulls.
\end{conjecture}

We disprove this conjecture by presenting an infinite class of examples with $n$ points and $m=n+1$ convex hulls:

\begin{theorem}\label{thm:counter}
    For every $n\ge 6$, there exists a convex hull thrackle on $n$ points with $n+1$ convex hulls.
\end{theorem}

In \cite{agoston2022orientation} they note that an interesting special case is when $P$ is in convex position. Indeed, this case is a generalization of outerplanar linear thrackles (for thrackles this specific case was not so important as any one of the two conditions in itself already allows the proof method of Perles to be applied). For this case, we can prove that their conjecture holds:

\begin{theorem}\label{thm:convex}
	If $P$ is a set of $n$ points in convex position then any convex hull thrackle on $P$ has at most $n$ convex hulls.
\end{theorem}

From the proof it also follows that the extremal examples in Theorem \ref{thm:convex} are extensions of extremal examples of linear thrackles, see Claim \ref{claim:convexextr}. Some examples of linear thrackles with $n$ points are the following. Take an odd number of points evenly spaced on a circle where each point is connected by a segment to the two points that are almost opposite to it. Or take points in convex position and connect one point to every other point by a segment and also connect by a segment its two neighbors on the convex hull of the point set.

Even though there is a counterexample for the above conjecture, convex hull thrackles are a pretty natural extension of linear thrackles and worth investigating further. In particular, as far as our construction is concerned, the conjecture may be already true if we replace the upper bound $n$ with $n+1$.
Our third result is that replacing $n$ with $2n$ is enough:

\begin{theorem}\label{thm:2n}
A convex hull thrackle on $n$ points can have at most $2n$ convex hulls.
\end{theorem}

Another direction would be to find a proper adjustment of the definition that excludes our counterexample and for which the original conjecture may be true. We discuss some variants of the definition in Section \ref{sec:disc}.

\bigskip
We finish the Introduction by mentioning a related definition of Asada et al. \cite{asada2016reay}. They also extend linear thrackles to convex hulls. On a given points set $P$ a set of convex hulls $C(\cS)$ is a \emph{thrackle of convex sets} if there exists a finite set $W\supseteq P$ such that for each pair of convex hulls $C_i\ne C_j$ we have $|C_i\cap C_j\cap W|=1$. They conjecture that in this case we still have at most $n$ convex hulls. Note that given a linear thrackle (of segments) we can add to $P$ the set of intersection points of segments to get the required $W$. Thus this definition is also a natural extension of linear thrackles. They prove their conjecture in special cases only. In particular they show that the number of convex hulls is at most $|W|$. We do not see any direct connection between this definition and the definition of convex thrackles we worked with.

\section{Points in convex position}

We first prove a simple special case, note that for this the points are not necessarily in convex position.

\begin{claim}\label{claim:1point}
	A convex hull thrackle on $n$ points in which every set $S\in \cS$ contains a fixed point $p\in P$ can have at most $n-1$ convex hulls.
\end{claim}

\begin{proof}[Proof of Claim \ref{claim:1point}.]
	Let $p\in P$ be the point such that $\{p\}\in S$ for every $S\in \cS$. If for some $q\ne p$ we have that $\{p,q\}\in \cS$ then $q$ cannot be part of any other convex hull as there are no convex hulls containing each other. Thus if there are $k$ such points then there are at most $k$ such segments in $C(\cS)$ and there is no other convex hull which is a segment. Convex hulls that are not segments contain at least two segments containing $p$, on the other hand every such segment is in at most two convex hulls. Therefore the number of non-segment convex hulls is at most the number of points not counted earlier, that is, at most $n-1-k$. Altogether we have at most $n-1$ convex hulls, as claimed.	
\end{proof}

In the rest of this section, we always assume that $P$ is in convex position.

\begin{proof}[Proof of Theorem \ref{thm:convex}]
From now on let $\cS$ such that $C(\cS)$ is a convex hull thrackle on $P$. We take a minimal counterexample, i.e., where $|P|=n$ is as small as possible, $|\cS|=m>n$ and such that $\sum_{S\in\cS}|S|$ is as small as possible. This implies that $m=n+1$. Whenever we say that a convex hull thrackle is smaller than another one we refer to this ordering. If there is a set $S\in \cS$ with $|S|=1$ then $|\cS|=1$, a contradiction, therefore every set in $\cS$ has size at least two.
The proof is a consequence of a series of lemmas.
%

\begin{lemma}\label{lem:boundary}
	There cannot be a boundary edge of $Conv(P)$ which is a proper subset of a convex hull from $C(\cS)$.
\end{lemma}
\begin{proof}
    Assume on the contrary and let $Conv(p_1,p_2)$ be such a boundary edge ($p_1,p_2\in P$) and let $S\in \cS$ contain both of $p_1$ and $p_2$ and a further vertex $p_3$. We claim that we can remove one of $p_1,p_2$ from $S$ to get a smaller counterexample, which is a contradiction. 
    Indeed, assume on the contrary that if we remove $p_1$ from $S$ then the resulting family of convex hulls is not a convex hull thrackle. Only the first property can fail, that is, there is a set $S_1\in \cS$ which is disjoint from $Conv(S\setminus\{p_1\})$. Removing $p_2$ shows that there is also a set $S_2\in \cS$ disjoint from $Conv(S\setminus\{p_2\})$. Using that $P$ is in convex position and $p_1$ and $p_2$ are consecutive on the hull, we get that $S_1\cap S=\{p_1\}$, $S_2\cap S=\{p_2\}$ and that $Conv(S_1)\setminus\{p_1\}$ and $Conv(S_2)\setminus\{p_2\}$ are in different connected components of $Conv(P)\setminus Conv(S)$ (here we used that $p_3$ exists as otherwise there would be only one connected component). Thus $Conv(S_1)$ and $Conv(S_2)$ are disjoint, contradicting that originally we had a convex hull thrackle.
    See Figure \ref{fig:convex-boundary}.
\end{proof}

\begin{corollary}\label{cor:boundary}
	No boundary edge of $Conv(P)$ can be in two convex hulls from $C(\cS)$.
\end{corollary}

\begin{figure}
	\begin{center}
		\includegraphics[scale=0.8]{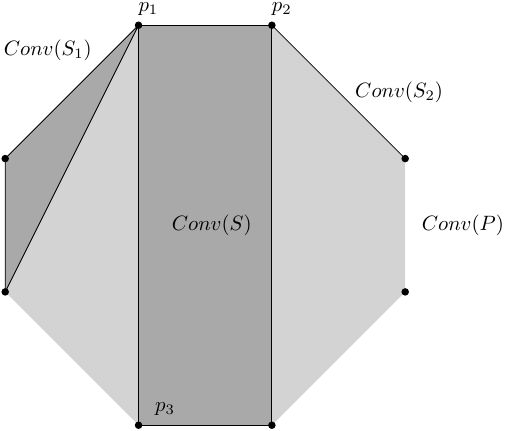}
		\caption{Convex hull containing a boundary edge.}
		\label{fig:convex-boundary}
	\end{center}
\end{figure}

\begin{lemma}\label{lem:n-1}
	No segment $Conv(\{p_1,p_2\})$ with $p_1,p_2\in P$ which is not a boundary edge in $Conv(P)$ is contained in two convex hulls from $C(\cS)$.
\end{lemma}
\begin{proof}
	Assume on the contrary and let $C_1,C_2\in C(\cS)$ be two convex hulls that contain both of $p_1$ and $p_2$. Every other convex hull must be disjoint from the interior of $e=Conv(\{p_1,p_2\})$. The segment $e$ splits $Conv(P)$ into two convex sets. Except for $C_1,C_2$, every other convex hull is strictly on one side of $e$ (by strictly we mean that it is disjoint from the other component of $Conv(P)\setminus\{e\}$) and contains at most one of $p_1,p_2$. 
	
	We claim that for both sides of $e$ there is a convex hull strictly on that side. Assume on the contrary, then wlog. there is no convex hull strictly on the left side of $e$. Then $\cS\setminus\{C_1\}$ restricted to the right side (actually only $C_2$ needs to be restricted, the rest are already on the right side) must be a convex hull thrackle on at most $n-1$ points ($n-1$ comes from the fact that there are points of $P$ on both sides of $e$ besides $p_1,p_2$ as it is not a boundary edge of $Conv(P)$) as by restricting to the right side no containments could be introduced. Thus by the minimality of our example we have that $m-1\le n-1$ and thus $m\le n$, a contradiction.
	
	Thus a convex hull strictly on the right and a convex hull strictly on the left exist, they can only meet at points $p_1,p_2$, so every convex hull must contain either $p_1$ or $p_2$ while only $C_1$ and $C_2$ contains both. Therefore if a convex hull on one side contains $p_1$, then on the other side all the additional convex hulls must contain $p_1$, and the same is true the other way. Hence one of $p_1$ or $p_2$ will be contained in all of the convex hulls.
	
	Now by Claim \ref{claim:1point} we get that $\cS$ has at most $n-1$ elements, a contradiction.
\end{proof}

\begin{corollary}\label{convex-no2}
	No pair of points of $P$ can be in two convex hulls from $C(\cS)$.
\end{corollary}
\begin{proof}
    The corollary is a direct consequence of Corollary \ref{cor:boundary} and  Lemma \ref{lem:n-1}.
\end{proof}

Corollary \ref{convex-no2} guarantees that whenever we replace a set $S$ in $\cS$ with a subset of $S$ of size $2$ then the first condition of being a convex hull thrackle will still hold. This will be used repeatedly in the remainder of this section.

\begin{lemma}
	If there are at least $3$ convex hulls from $C(\cS)$ containing a point $p\in P$ then there are two whose intersection is  $\{p\}$.
\end{lemma}
\begin{proof}
	Let $C_1=Conv(S_1),C_2=Conv(S_2),C_3=Conv(S_3)$ be these convex hulls. Let $p,p_2,\dots p_n$ be the points of $P$ ordered counterclockwise on the boundary of $Conv(P)$. Now let $I_i$ be the minimal length interval of $p_2,\dots p_n$ that contains $S_i\cap P\setminus\{p\}$. It is easy to see that if $I_1\cap I_2\cap I_3\ne \emptyset$ then $C_1\cap C_2\cap C_3\notin P$, a contradiction. Otherwise, if, $I_1\cap I_2\cap I_3\ne \emptyset$ then two of these intervals must be disjoint, and then the intersection of the two corresponding $C_i$'s is $\{p\}$.	
\end{proof}

\begin{figure}
	\begin{center}
		\includegraphics[scale=0.8]{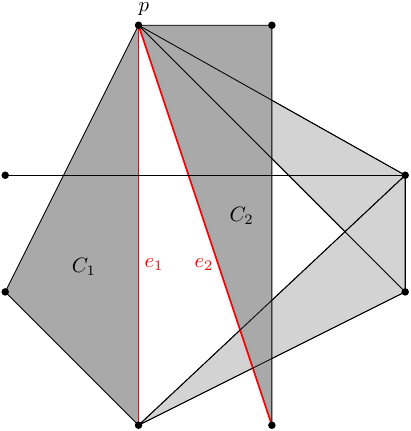}
		\caption{$C_1$ can be replaced with $e_1$ and $C_2$ with $e_2$ to get a smaller convex hull thrackle.}
		\label{fig:convex-replace}
	\end{center}
\end{figure}

\begin{lemma}\label{lem:onlyp}
	If there are two convex hulls whose only intersection is $\{p\}$ for some $p\in P$ then both of them are segments.
\end{lemma}
\begin{proof}
	
Indeed, otherwise we can replace them with their appropriate boundary edges incident to $p$ (the edge closer to the other convex hull) and it is easy to see that they must still intersect every other convex hull, thus it remains to be a convex hull thrackle. Also, it would be a smaller one, a contradiction. See Figure \ref{fig:convex-replace}.	
\end{proof}

\begin{figure}
	\begin{center}
		\includegraphics[scale=0.8]{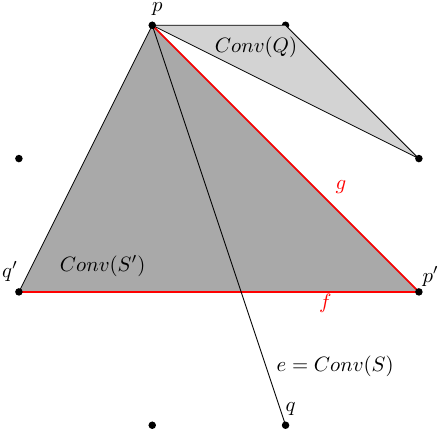}
		\caption{$S'$ can be replaced with either $f$ or $g$ to get a smaller convex hull thrackle.}
		\label{fig:convex-notonlysegment}
	\end{center}
\end{figure}

\begin{lemma}
	There cannot be a set $S\in \cS$ of size two and a set $S'\in \cS$ of size $\ge 3$ that have a nonempty intersection. 
\end{lemma}

\begin{proof}
	Assume on the contrary that $S\cap S'$ contains a point $\{p\}$, $p\in P$.
	By Lemma \ref{lem:onlyp} $\{p\}$ must be a proper subset of $Conv(S)\cap
 Conv(S')$.
 
	Let $S=\{p,q\}$. Let $f=Conv(\{p',q'\})$ be the unique edge on the boundary of $Conv(S')$ whose interior intersects the segment $e=Conv(S)$. If replacing $S'$ with $\{p',q'\}$ we get a convex hull thrackle then it is a smaller counterexample, a contradiction. Otherwise, if it is not a convex hull thrackle then there must be a convex hull $Conv(Q)\in C(\cS)$ such that $Q\cap f=\emptyset$. As $Conv(Q)\cap Conv(S)\neq \emptyset$ we get that $p\in Q$. Wlog. $p'$ and $Q\setminus\{p\}$ are contained in the same connected component of $Conv(P)\setminus S$. Let $g=\{p,p'\}$. As $Conv(Q)$ is disjoint from $f$, $Conv(Q)\setminus\{p\}$ must be disjoint from $g$ and thus be in one of the connected components of $Conv(P)\setminus g$, precisely in the one that does not contain $q$. In this case, we replace $S'$ with $g$. As every other convex hull must intersect both $Conv(Q)$ and $Conv(S)$, they all must intersect the segment $g$ as well, and so this is still a convex hull thrackle, but a smaller one, a contradiction. See Figure \ref{fig:convex-notonlysegment}.		
\end{proof}

The lemmas imply that each point of $P$ is either only in sets from $\cS$ of size $2$ or only in sets from $\cS$ of size $\ge 3$, moreover, in the latter case, there can be at most two such sets for each point. Thus $P$ can be split into two disjoint subsets $P=P_2\cup P_3$, such that the sets $S\in \cS$ that contain $2$ points of $P$ contain only points from $P_2$ while the sets $S \in \cS$ that contain at least $3$ points of $P$ contain only points from $P_3$. Moreover for every $p\in P_3$ at most two sets from $\cS$ contain $p$. Thus there are at most $2/3|P_3|$ sets of size $\ge 3$ while for $P_2$ we can apply the result about (convex) linear thrackles (see Corollary \ref{cor:linear} for a proof) to conclude that there are at most $|P_2|$ sets of size $2$, altogether there are at most $n$ sets in $\cS$, a contradiction.
\end{proof}

\begin{claim}\label{claim:convexextr}
    If a convex hull thrackle on $n$ points in convex position has $n$ convex hulls, then there is an underlying linear thrackle, i.e., there is an injection $S\Rightarrow S': S'\subseteq S, |S'|=2$ from $\cS$ such that the image of $\cS$ is a linear thrackle.
\end{claim}

\begin{proof}
Let us have now an extremal example, i.e. a convex hull thrackle with $m=n$ convex hulls. Going through the proof of Theorem \ref{thm:convex} we see that we did not directly use that our convex hull thrackle is a counterexample, instead we always replaced some non-segment convex hull with its proper subset contradicting its minimality.

Thus repeating the algorithm steps in the proof of Theorem \ref{thm:convex} we repeatedly replace convex hulls with smaller convex hulls to get another convex hull thrackle with the same number of points and convex hulls as before. At the end, according to the computation at the end of the proof of Theorem \ref{thm:convex}, if there are non-segment convex hulls left then we have less than $n$ convex hulls, a contradiction. Thus at the end every convex hull must be a segment, finishing the proof.

There is one exception when in the replacement we used the fact that $m=n+1$. This was in Lemma \ref{lem:n-1} where we used that $m=n+1$ to show that there are convex hulls strictly on both sides of the segment $e$. Let us modify slightly the argument from there. Now we work with a convex hull thrackle with $m=n$ and assume on the contrary that on the left side there is no convex hull. If the left side (including $p_1,p_2$) has size at least $n_1\ge 4$ then the restriction of $\cS\setminus\{C_1\}$ is a convex hull thrackle on at most $n-2$ points and with $m-1=n+1-1=n-2+1$ thrackles, a contradiction by Theorem \ref{thm:convex}.

\begin{figure}
	\begin{center}
		\includegraphics[scale=1]{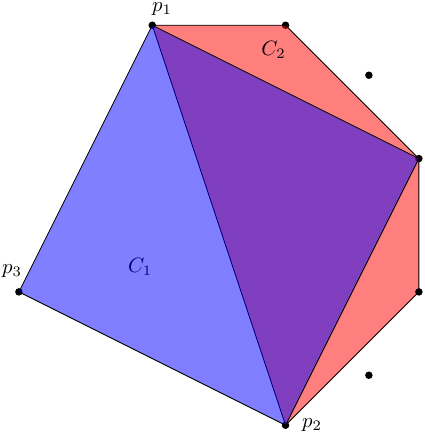}
		\caption{Lemma \ref{lem:n-1} modified for extremal examples.}
		\label{fig:lemmod}
	\end{center}
\end{figure}

If the restrictions of $C_1$ and $C_2$ on the right side do not contain each other then $\cS$ restricted to the right side is a convex hull thrackle on $n-1$ points and $m$ convex hulls, a contradiction by Theorem \ref{thm:convex}. If the restrictions of $C_1$ and $C_2$ on the left side do not contain each other then $n_1\ge 4$ and we are done. Thus on each side the restrictions of $C_1$ and $C_2$ must contain each other and $n_1=3$. Thus if $p_3$ is the unique vertex to the left of $e$ then wlog. $C_1$ contains $p_3$ while $C_2$ does not contain $p_3$ while $C_2\supsetneq C_1\setminus\{p_3\}$. Now applying the proof of Lemma \ref{lem:boundary} on the restriction of $\cS\setminus\{C_1\}$ on the right side we get that we can remove $p_1$ or $p_2$ from $C_2$ to get a convex hull thrackle. This together with $C_1$ must be also a convex hull thrackle. See Figure \ref{fig:lemmod}.
\end{proof}

\begin{figure}
	\begin{center}
		\includegraphics[scale=1]{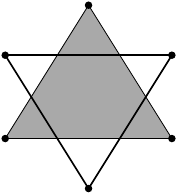}
		\caption{A convex hull thrackle in which no triangle can be replaced with a segment inside it.}
		\label{fig:constr-nocont}
	\end{center}
\end{figure}

We note that for non-extremal convex thrackles it is not always true that there is an underlying linear thrackle, see, e.g., the construction for Theorem \ref{thm:counter} or Figure \ref{fig:constr-nocont} for such an example with points in convex position.
\section{General case lower bound construction}

\begin{figure}
	\begin{center}
		\includegraphics[scale=1]{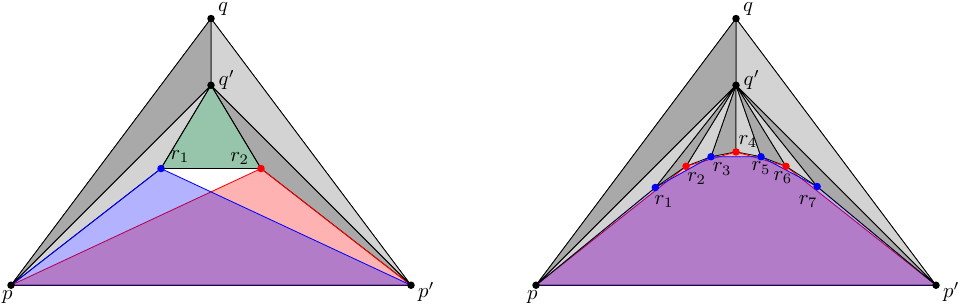}
		\caption{Convex hull thrackle with $n+1$ convex hulls for $n=6$ and $n=11$.}
		\label{fig:construction}
	\end{center}
\end{figure}

\begin{proof}[Proof of Theorem \ref{thm:counter}]
See Figure \ref{fig:construction} for an illustration. The points $p,p',q$ form a triangle which contains $q'$ and there are additional points $r_1, \dots, r_{n-4}$ inside $ppq'\Delta $ such that $r_1,\dots, r_{n-4},p',p$ are in convex position in this order positioned such that they can be seen from $q'$ also in this order. There are $n-1$ triangular convex hulls incident to $q'$, one formed by each containment minimal angle at $q'$ determined by the remaining points.
There are two more convex hulls, both contain $p,p'$. One contains in addition the points $r_i$ with $i$ even and the other the points $r_j$ with $j$ odd.
 
It is easy to check that this is a convex hull thrackle on $n$ points and with $n+1$ convex hulls.
\end{proof}

\section{General case upper bound proof}

\begin{proof}[Proof of Theorem \ref{thm:2n}.]
If there is a set $S\in \cS$ with $|S|=1$ then $|\cS|=1$, we are done, therefore we can assume that every set in $\cS$ has size at least two. For a convex hull thrackle $C(\cS)$ on point set $P$ we define its \emph{boundary diagram} $D(\cS)$ (or simply $D$ when $\cS$ is clear from the context) the following way: we replace each convex hull of with at least $3$ vertices with the set of its boundary segments. Notice that each such segment is joining two vertices of the graph. If the convex hull was a segment, then we replace it with $3$ copies of this segment. This way we get a multiset of segments. If on a pair of points there are $k$ segments, then we regard this as one segment with weight $k$. Thus $D$ is a weighted (not multi)set of segments connecting points of $P$. Notice that by definition of a convex hull thrackle, a segment $e$ of $D$ that is also present in $C(\cS)$ as a convex hull (i.e., $e=Conv(S)$ with $S\in\cS$, $|S|=2$) has weight exactly $3$ (as by definition no other $S'$ can contain $S$), while the rest of the segments of $D$ can have weight at most $2$ (as by definition a segment can be the boundary segment of at most two convex hulls).
 
Since every convex hull has been replaced with at least $3$ segments when defining the boundary diagram, it is enough to show that the sum of weights of segments in $D$ is at most $6n$, where $n=|P|$. 
 
 See Figure \ref{fig:def-leftto} for the next set of definitions.
 
 \begin{definition}
 	Given points $a,p,b$, the angle $apb\angle$ is defined to be the angle at $p$ we get when we rotate the vector $\vec{pa}$ counterclockwise (around $p$) to make it coincide with $\vec{pb}$.
 	
	An object (a point, a set of points, or a convex hull) is \emph{to the left}  of a vector $\vec{pa}$ if for every point $b$ of the object we have $0\degree<|apb\angle|\le 180\degree$. $p$ is also said to be left to $\vec{pa}$. 
 \end{definition}

\begin{figure}
	\begin{center}
		\includegraphics[scale=1]{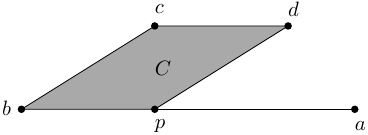}
		\caption{ $|apc\angle|=90\degree$. $C,b,c,d$ and $p$ are left to $\vec pa$ but $a$ is not. $dpb\angle$ is a wedge at $p$ whose left side is $pb$. $C$ witnesses that $pa$ is a leftie from $p$.}
		\label{fig:def-leftto}
	\end{center}
\end{figure}

\begin{definition}
	Given a point $p\in P$ and a point set $S\in \cS$ for which $p$ is a vertex of $Conv(S)$. Let $a$ and $b$ be the vertices of $Conv(S)$ next to $p$ on the boundary of $Conv(S)$ such that $b$ is left to $\vec{pa}$ (where $a=b$ is allowed when $Conv(S)$ is a segment). The $apb\angle$ is called the \emph{wedge} at $p$. We also say that $pb$ is the \emph{left side of the wedge}.	
\end{definition}

 \begin{definition}
  A segment $pa$ of $D$ is called a \emph{leftie from $p$}, if there is a point set $S\in \cS$ for which $S$ (equivalently $Conv(S)$) is to the left of $\vec{pa}$. We say that such an $S$ (also $Conv(S)$) witnesses that $pa$ is a leftie from $p$. Otherwise, it is called a \emph{non-leftie from $p$}.
 \end{definition}

We note that being a leftie could be defined to any pair of points the same way, not just for pairs that form a segment of $D$, but we do not need this. Also, usually the convex hull $C$ witnessing that $pa$ is a leftie from $p$ will be incident to $p$.

 \begin{lemma}\label{lem:lefties}
 	No segment can be a leftie from both of its endpoints. 
 \end{lemma}

\begin{proof}
	Assume on the contrary. Then the two convex hulls witnessing that the segment is a leftie from both of its endpoints must be disjoint, a contradiction. See Figure \ref{fig:no2lefties}.
\end{proof} 

 \begin{figure}
 	\begin{center}
		\includegraphics[scale=1]{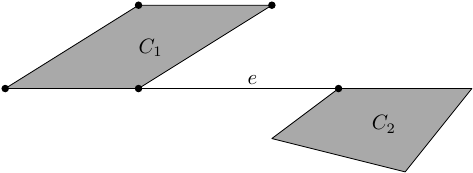}
 		\caption{If $e$ is a leftie from both of its endpoints then the witnessing convex hulls $C_1$ and $C_2$ are disjoint.}
 		 \label{fig:no2lefties}
 	\end{center}
 \end{figure}

\begin{observation}\label{obs:no2nonleftie} For any two wedges at $p$ that intersect only in $p$, their left sides cannot be both non-lefties from $p$.
\end{observation}

\begin{proof}
	Call one of these sides $pa$, the other $pb$. If $0\degree< |bpa\angle|\leq 180\degree$ then  $pb$ is a leftie from $p$, but if $0\degree<|apb\angle|\leq 180\degree$ then $pa$ is a leftie from $p$.
\end{proof}

The restriction to segments of the concept of lefties and the previous two statements about it are the tools in the proof of Perles mentioned in the Introduction \cite{pach2011conway} that linear thrackles have at most $n$ segments, we show the proof for the sake of completeness:

\begin{corollary}\label{cor:linear}
	If every convex hull is a segment (i.e., we have a linear thrackle) then there are at most $n$ segments.
\end{corollary}
\begin{proof}
	If every convex hull is a segment then by Observation \ref{obs:no2nonleftie} 
    no two segments incident to the same vertex $p$ can be non-lefties from $p$ (note that the left side of a segment is the segment itself). On the other hand by Lemma \ref{lem:lefties} every segment is a non-leftie from one of its endpoints. These together imply that we have at most as many segments as points of $P$.
\end{proof}

Back to the proof of Theorem \ref{thm:2n}, in the general case we also try to do something similar to what is done in the proof of Corollary \ref{cor:linear}. More precisely, it is enough to prove that for each vertex $p$ the sum of the weight of segments of $D$ incident to $p$ that are non-lefties from $p$ is at most $6$. Indeed, by Lemma \ref{lem:lefties} each segment of the boundary diagram has to be non-leftie from at least one of its endvertices so if each vertex has at most a weight of $6$ non-lefties from it, then the whole boundary diagram can have at most $6n$ weight of segments, hence the convex hull thrackle can contain at most $2n$ convex hulls. Thus, the following lemma concludes the proof of Theorem \ref{thm:2n}:

 \begin{lemma}\label{lem:6}
 For every vertex $p$ the sum of the weight of segments incident to $p$ that are non-lefties from $p$ is at most $6$.
 \end{lemma}

 \begin{proof}
 The proof of the lemma is split into a few cases.

 \begin{itemize}
     \item{Case 1:} There exists a segment $pa$ in $D$ which is also a convex hull in $\cS$ (thus has weight $3$) and $pa$ is a non-leftie from $p$.
     
	In this case for any $q$ where $pq$ is a non-leftie from $p$, we know that $0\degree \le |apq\angle|< 180\degree$, as otherwise the convex hull $pa$ would witness that $pq$ is a leftie. In addition, this angle cannot have size  $0\degree$, since then either $3$ vertices would be collinear, or the convex hull $pa$ would be contained in another convex hull from $C(\cS)$, both of which are forbidden. That is, $q$ is left of $pa$.
     
	Thus, we cannot have another weight $3$ non-leftie segment from $p$ or two non-leftie segments from $p$ that are on the boundary of the same convex hull, because that hull would then witness $pa$ being a leftie from $p$. So for each non-leftie segment incident to $p$ besides $pa$ there must be another leftie segment which is part of the boundary of the same convex hull.
     
    From now on, we think of these convex hulls as if they were the wedges they define.
     
    To rephrase what we have shown already, each non-leftie segment at $p$ is to the left to $pa$ while the other segment incident to $v$ on the boundary of the same convex hull is not to the left to $pa$, thus the corresponding wedge contains a non-zero part of the line defined by $pa$. We also know that no three wedges can overlap by the definition of a convex hull thrackle. Hence at most one of these wedges can contain the halfline $pa$, and at most two the complement halfline on the line defined by $pa$. This gives at most $3$ wedges that can contribute weight $1$ to the non-leftie segments incident to $p$. Together with $pa$, this is at most weight $6$ of segments, as required, finishing this case. See Figure \ref{fig:case1}
     
     \begin{figure}
     \begin{center}
     \includegraphics[scale=1]{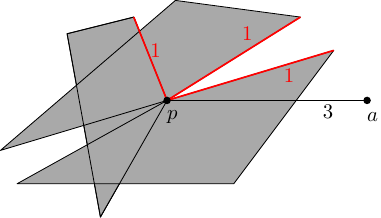}
     \caption{The maximum sum of weights of non-leftie segments at $p$ is $6$ in Case 1.}
     \label{fig:case1}
     \end{center}
     \end{figure}

     From now on we can and will assume that there is no segment incident to $p$ in $D$ which is also a convex hull in $\cS$ and is a non-leftie from $p$. 
     
     \item{Case 2:} There is a wedge $apb\angle$ whose both sides are non-lefties from $p$. 
     
	Denote by $a'$ and $b'$ the reflection of $a$ and $b$ to $p$, respectively. Note that it is impossible for a convex angle at $p$ (thus also a wedge) to contain more than $2$ of $\{pb,pa,pb',pa\}$. Moreover, using that $pa$ and $pb$ are non-lefties, it is easy to see that every wedge $w$ that has a non-leftie side fits into one of the following Types (listing also its contribution to the sum of weights of non-leftie segments of $D$ at $p$):

     \begin{enumerate}
     \item $w$ contains $pa$ but not $pb'$. Only the left side of $w$ can be a non-leftie. Contributes weight $\le 2$.
     
     \item $w$ contains $pb'$ but not $pa$. Contributes weight $\le 1$.
     
     \item $w$ contains both $pb$ and $pa'$. Contributes weight $\le 2$.
     
     \item $w$ contains both $pa$ and $pb'$. Contributes weight $\le 2$.
    \end{enumerate}
     
     \begin{figure}
	\begin{center}
		\includegraphics[scale=1]{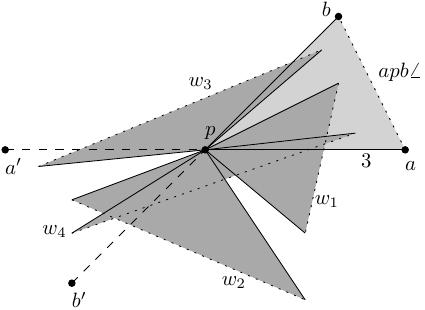}
		\caption{The possible types in Case 2: wedge $w_i$ is of Type $i$ for $i\in [4]$.}
		\label{fig:case2}
	\end{center}
	\end{figure}

     We cannot have more than one Type $4$ wedge (as these and $\angle apb$ all contain $pa$).
     
     \begin{itemize}

     	\item{Case 2(a):} There is no Type $4$ wedge at $p$.
     
	In this case, we can have at most one Type $1$ wedge, at most two Type $2$, and at most one Type $3$ wedge. Their maximal weight contribution together is thus $6$. If the Type $1$ or Type $3$ wedge does not exist then together with the contribution of $\angle apb$ we have a maximum sum of weight of $6$.

    Otherwise, notice that the Type $1$ and Type $3$ wedges must be disjoint apart from point $p$. Then by Observation \ref{obs:no2nonleftie} their left sides cannot be both non-lefties and so at least an additional weight of $1$ is lost.

    Thus, to have a total weight of $7$, all the Type $2$ wedges must exist too.
    Let the left side of the Type $1$ wedge be $pc$, and the left side of the Type $3$ wedge be $pd$. Exactly one of them is a leftie. 
    
    The Type $3$ wedge must be disjoint apart from $p$ from one of the Type $2$ wedges, as they cannot both contain $pd$. Since they are Type $2$, they do not contain $pa$ either, so the disjoint wedge must witness $pd$ is a leftie.

    So $pc$ must be a non-leftie. So the Type $3$ wedge must contain $pc'$. Then there is a Type $2$ wedge that does not contain $pc'$. But then the Type $1$ wedge witnesses that the right side of this Type $2$ wedge is a leftie. So another weight of non-leftie segments must be lost, thus the weight contribution of all non-leftie segments is at most $4$ excluding $\angle apb$, which together with the contribution of $\angle apb$ gives a maximum sum of weight of $6$.
	
\item{Case 2(b):} There is one Type $4$ wedge at $p$.

	In this case, we have one Type $4$, at most one Type $2$, at most one Type $3$, and no Type $1$ wedges. Their maximal weight contribution together is thus $5$. If any of them does not exist then together with the contribution of $\angle apb$ we have a maximum sum of weight of $6$.

    Otherwise, if the Type $4$ and Type $3$ wedges are disjoint apart from $p$, then again by Observation \ref{obs:no2nonleftie}, their left sides cannot be both non-lefties and so at least an additional weight of $1$ is lost.

    Otherwise, the Type $3$ and Type $4$ wedges must have at least a common segment. Then the Type $2$ wedge and the Type $3$ wedge must be disjoint apart from $p$ as there are no triple intersections. However, the Type $2$ wedge does not contain $pa$ and thus it must witness that the left side of the Type $3$ wedge is a leftie and so at least an additional weight of $1$ is again lost.

    Thus the weight contribution of all non-leftie segments is always at most $4$ excluding $\angle apb$, which together with the contribution of $\angle apb$ gives a maximum sum of weight of $6$.

\begin{figure}
	\begin{center}
		\includegraphics[scale=1]{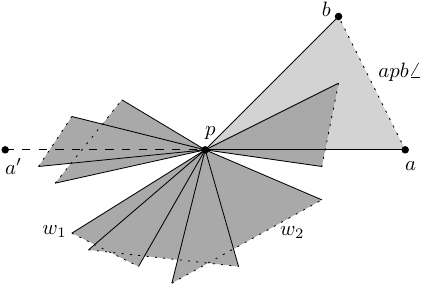}
		\caption{Case 3: $w_2$ witnesses that both sides of $w_1$ are lefties.}
		\label{fig:case3}
	\end{center}
\end{figure}

     \end{itemize}    	
     \item{Case 3:} Every wedge has a leftie side. In addition to that, there is a wedge whose left side is a leftie and whose right side is a non-leftie.

     Let $apb\angle$ be a wedge such that $pb$ is a leftie from $p$, and $pa$ is a non-leftie from $p$. Denote by $a'$ the reflection of $a$ to $p$.     

	 Since $pa$ is a non-leftie, every other wedge at $p$ must either be contained entirely in $a'pa\angle$, or must contain $pa'$ or $pa$.

     At most two such wedges can contain $pa'$, and at most one wedge can contain $pa$. If there are at most two wedges entirely in $a'pa\angle$ with a non-leftie side then the sum of the weight of all wedges at $p$ is at most $6$, as required. Otherwise, there are at least $3$ wedges entirely in $a'pa\angle$ having a non-leftie side. Notice that among these there must be two wedges that are disjoint apart from $p$ (as there are no triple intersections apart from $p$). Since these two wedges are contained in an angle $\leq 180\degree$, the wedge on the left witnesses that both sides of the one on the right are lefties, a contradiction. See Figure \ref{fig:case3}.        

     \item{Case 4:}
     All wedges with a non-leftie side have their non-leftie side on the left side.
     
     In this case by Observation \ref{obs:no2nonleftie} we cannot have $2$ disjoint wedges with a non-leftie side at $p$. As among any $4$ wedges at $p$, two must be disjoint (as there are no triple intersections apart from $p$), we can have at most $3$ wedges at $p$, thus their weight is at most $3<6$, as required.
 \end{itemize}
We have exhausted all cases and proved that in each case the sum of the weight of the wedges at $p$ is at most $6$, as required.
 \end{proof}
 This finishes the proof of Theorem \ref{thm:2n}.
\end{proof}

\section{Variations on the definition}\label{sec:disc}

We briefly investigate the necessity of the assumptions in the definition of convex hull thrackles. Allowing $C_1\cap C_2= \emptyset$ makes a straight line drawing of the complete graph allowed, having $n\choose 2$ segments. Not requiring $C_1\cap C_2\cap C_3\subset P$, the following construction shows that we can have $(n/3)^3-o(1)$ convex hulls. Take $n$ points in convex position, split them into $3$ equal intervals and for each triple of points, one from each interval, we take its convex hull.

\begin{figure}
	\begin{center}
		\includegraphics[scale=1]{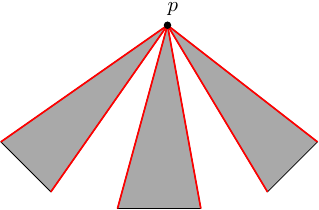}
		\caption{The construction with $n-1$ segments and $\lfloor (n-1)/2\rfloor$ triangles when containments are allowed.}
		\label{fig:constr-cont}
	\end{center}
\end{figure}

Regarding the assumption that there is no containment between the convex hulls, allowing $C_1\subset C_2$ we can get $\lfloor \frac{3}{2}(n-1)\rfloor$ convex hulls, as observed by Gossett \cite{pc}, who made this remark in order to disprove the version of Conjecture \ref{conj} that originally appeared in \cite{agoston2022orientation}, in which containments were allowed. Gossett's construction is the following. Take $n$ points in convex position, connect one of them, $p$, to all others by segments. Add also every second triangle incident to $p$ determined by segments consecutive in their linear order around $p$. See Figure \ref{fig:constr-cont}. On the other hand, our proof works in this case too except for Case $1$, in which case instead of sum of weight at most $6$ we can only guarantee a sum of weight at most $7$. Thus, when we allow containments, this gives the upper bound $7n/3$ on the number of convex hulls. 
Furthermore, if alongside containments we also allow that the family of convex hulls is a multiset (i.e., the same convex hull can appear multiple times in our family of convex hulls), then doubling the segments of a star we can have at least $2n-2$ convex hulls, again even if the points are in convex position. On the other hand in Case $1$ we can guarantee a sum of weight at most $8$, which gives the upper bound $8n/3$.

Finally, what if we do not assume that the points are in general position, i.e.,  there can be $3$ or more points on a line? We still do not have a better lower bound than $n+1$. On the other hand, we can again modify the proof of Lemma \ref{lem:6} to work for points in non-general position. First, when defining the boundary diagram we need to use maximum length sides, i.e., even if a point lies on a side of some convex hull, in the diagram there is only a single segment associated with this side. With this modification the same proof works, except that in Case $1$ we yet again lose a bit, we can only guarantee a sum of weight at most $7$, which gives the upper bound $7n/3$. If we allow both points in non-general position, containments and multisets then we can still guarantee a sum of weight at most $8$, which gives the upper bound $8n/3$.

The details of these various modifications to Case $1$ of the proof of Lemma \ref{lem:6} are left to the interested reader.

\bibliographystyle{plainurl}
\bibliography{thrackles}

\begin{thebibliography}{10}

\bibitem{agoston2022orientation}
P{\'e}ter {\'A}goston, G{\'a}bor Dam{\'a}sdi, Bal{\'a}zs Keszegh, and
  D{\"o}m{\"o}t{\"o}r P{\'a}lv{\"o}lgyi.
\newblock Orientation of convex sets.
\newblock {\em arXiv preprint arXiv:2206.01721}, 2022.

\bibitem{asada2016reay}
Megumi Asada, Ryan Chen, Florian Frick, Frederick Huang, Maxwell Polevy, David
  Stoner, Ling~Hei Tsang, and Zoe Wellner.
\newblock {On Reay's relaxed Tverberg conjecture and generalizations of
  Conway's thrackle conjecture}.
\newblock {\em arXiv preprint arXiv:1608.04279}, 2016.

\bibitem{cairns2000bounds}
Grant Cairns and Yury Nikolayevsky.
\newblock Bounds for generalized thrackles.
\newblock {\em Discrete \& Computational Geometry}, 23:191--206, 2000.

\bibitem{cairns2012outerplanar}
Grant Cairns and Yury Nikolayevsky.
\newblock Outerplanar thrackles.
\newblock {\em Graphs and Combinatorics}, 28:85--96, 2012.

\bibitem{erdos1946sets}
Paul Erd{\H o}s.
\newblock On sets of distances of n points.
\newblock {\em The American Mathematical Monthly}, 53(5):248--250, 1946.

\bibitem{fulek2011computational}
Radoslav Fulek and J{\'a}nos Pach.
\newblock {A computational approach to Conway's thrackle conjecture}.
\newblock {\em Computational Geometry}, 44(6-7):345--355, 2011.

\bibitem{fulek2019thrackles}
Radoslav Fulek and J{\'a}nos Pach.
\newblock Thrackles: An improved upper bound.
\newblock {\em Discrete Applied Mathematics}, 259:226--231, 2019.

\bibitem{goddyn2017bounds}
Luis Goddyn and Yian Xu.
\newblock On the bounds of conway’s thrackles.
\newblock {\em Discrete \& Computational Geometry}, 58:410--416, 2017.

\bibitem{pc}
Ian Gossett.
\newblock personal communication.

\bibitem{Lovasz1997}
L.~Lovász, J.~Pach, and M.~Szegedy.
\newblock {On Conway's Thrackle Conjecture}.
\newblock {\em Discrete and Computational Geometry}, 18(4):369--376, 1997.
\newblock \href {https://doi.org/10.1007/pl00009322}
  {\path{doi:10.1007/pl00009322}}.

\bibitem{pach2011conway}
J{\'a}nos Pach and Ethan Sterling.
\newblock Conway's conjecture for monotone thrackles.
\newblock {\em The American Mathematical Monthly}, 118(6):544--548, 2011.

\bibitem{woodall1971thrackles}
Douglas~R Woodall.
\newblock Thrackles and deadlock.
\newblock {\em Combinatorial Mathematics and Its Applications}, 348:335--348,
  1971.

\bibitem{xu2021new}
Yian Xu.
\newblock A new upper bound for conway’s thrackles.
\newblock {\em Applied Mathematics and Computation}, 389:125573, 2021.

\end{thebibliography}
	
\end{document}